\RequirePackage{fix-cm}
\documentclass[smallextended]{svjour3}       
\smartqed  
\usepackage{mathrsfs}
\usepackage{amsmath}
\usepackage{amssymb}
\usepackage{bm}
\usepackage{lineno}
\usepackage{color}
\usepackage{esint}
\usepackage{mathdots}
\usepackage{graphicx}
\usepackage{hyperref}
\usepackage{subfig}
\usepackage{enumerate}

\begin{document}

\title{On locating the zeros and poles of a meromorphic function}


\author{Haotian Chen}


\institute{H. Chen\at
              Institute of Space Science and Technology, Nanchang University, Nanchang, 330031, China\\
Department of Atomic, Molecular and Nuclear Physics, University of Seville, Seville, 41012, Spain\\
              \email{haotianchen-ext@us.es}
}


\maketitle

\begin{abstract}
	On the basis of the generalized argument principle, here we develop a numerical scheme for locating zeros and poles of a meromorphic function.
	A subdivision-transformation-calculation scheme is proposed to ensure the algorithm stability. A novel feature of this algorithm is the ability to estimate the error level automatically.
	Numerical examples are also presented, with an emphasis on potential applications to plasma physics.
\end{abstract}

\section{Introduction}
\label{intro}
In this paper, we revisit the numerical method for locating zeros and poles of a meromorphic function $f(z)$ in a given region $\mathbb{D}$ on the complex plane, by employing the generalized argument principle. 
Historically, this question can be traced back to the pioneering work by Harry Nyquist in 1932 \cite{cit:nyquist}, which determined the stability of a dynamical system by searching the number of zeros of an analytic function in the upper-half plane. Although it has been widely used, the Nyquist stability criterion  can only provide the number of zeros, without information about their locations.
The modern argument principle approach to computing zeros of an analytic function was proposed by Delves and Lyness \cite{delves67}, in which a monic polynomial having the same zeros as the analytic function was introduced, with coefficients being calculated via Newton's identities. This procedure is, however, usually ill-conditioned. 
Later, the algorithm was modified to locate the zeros and poles of a meromorphic function \cite{cit:feb:09:18:47}.
In a series of papers by Kravanja \emph{et al.} \cite{kravanja98,kravanja99,kravanja99b}, the Delves-Lyness method was  systematically  extended by using the so-called formal orthogonal polynomials \cite{hildebrand87}.
This approach, as expected, has additional complication to generate the desired formal orthogonal polynomials for certain bilinear forms.
Based on a detailed  sensitivity analysis, here we propose a subdivision-transformation-calculation scheme to avoid ill-conditioning in numerical calculation. In particular,  the subdivision of region and the calculation of zeros and poles are separately carried out in different spaces.
A novel method of measuring the absolute error of the locations for zeros and poles is also given by a thorough error analysis.
Practical applications to the investigation of linear wave properties of various plasma waves are presented.


The paper is organized as follows. 
Section \ref{sec:feb:09:15:10} briefly reviews the mathematical background of finding zeros and poles via the generalized argument principle. 
In Sec. \ref{sec:feb:09:15:11}, we analyse the sensitivity of this algorithm, and give an explicit expression for the condition number, which leads us to a subdivision-transformation-calculation scheme for the algorithm stability control.
In Sec. \ref{sec:feb:09:15:12}, we prove a theorem for the singular pencil corrupted by noise, and apply it to the error estimate.
Section \ref{sec:feb:09:15:13} presents numerical examples. 
Conclusions are given in Sec. \ref{sec:feb:09:15:15}.

\section{Argument principle method for locating zeros and poles}
\label{sec:feb:09:15:10}
We consider a meromorphic function $f(z)$ in a closed complex domain $\mathbb{D}$, bounded by a Jordan curve $\mathcal{C}$. Assuming that $f(z)$ has $N$ zeros and poles within $\mathbb{D}$ but no zeros or poles on $\mathcal{C}$, then, from the well-known generalized argument principle, we have 
\begin{equation}
\label{eq:apr:06:15:26}
G_{k}\equiv\frac{1}{2\pi i}\ointctrclockwise_{C}\frac{f'(z)}{f(z)}g_{k}(z) dz =\sum_{j=1}^{N}g_{k}(a_{j})n_{j},
\end{equation}
where the summation is over all zeros and poles $a_{j}$ of $f(z)$ counted with their multiplicities $n_{j}$, and $g_{k}(z)$ is an analytic  function  in $\mathbb{D}$, which will be referred to as a probe function hereafter.
Given the values of $f(z), g_{k}(z)$ along $\mathcal{C}$, the fundamental idea of argument principle approach is to recover $(a_{j}, n_{j})$ from a series of $G_{k}$.
Furthermore, from Eq. (\ref{eq:apr:06:15:26}), it is evident that the efficiency of the algorithm depends strongly on efficient contour integration rules. 
However, a detailed discussion of the numerical contour integration method is beyond the scope intended for this work. 

The key issue for the argument principle approach is to choose suitable probe functions $g_{k}(z)$.
Although many sophisticated schemes have been suggested, here  we introduce a simple transformed probe functions 
\begin{equation}
\label{eq:jan:25:12:59}
g_{k}(\zeta) =\zeta^{\gamma_{0}+k\Delta\gamma}, \quad k=0,1,2,\cdots,
\end{equation}
where $\gamma_{0}, \Delta\gamma\in\mathbb{N}$, $\zeta$ denotes the new complex variable, and the transformation from $z$ to $\zeta$ is defined  symbolically as
\begin{equation}
\label{eq:feb:02:20:58}
	\zeta=T(z).
\end{equation}
The values of $\gamma_{0}, \Delta\gamma$ and the detailed expression of $T$ will be discussed later.
Accordingly, $\mathbb{D}$ is also transformed into the domain $\mathbb{D}_{\zeta}$ in $\zeta$-space.
With the new variable, the moments in $\zeta$-space can be expressed as
\begin{equation}
\label{eq:jan:25:16:01}
\bar{G}_{k}\equiv\frac{1}{2\pi i}\ointctrclockwise_{C}\frac{f'(\zeta)}{f(\zeta)}g_{k}(\zeta) d\zeta =\sum_{j=1}^{N}g_{k}(\zeta_{j})n_{j}.
\end{equation}

Once the moments $\bar{G}_{k}$ have been found, following  \cite{delves67,kravanja98}, we can construct the Hankel matrices $\bar{\mathbf{H}}_{0,N}, \bar{\mathbf{H}}_{1,N}$ as
\begin{displaymath}
	\bar{\mathbf{H}}_{0,N}=
	\left( \begin{array}{cccc}
			\bar{G}_{0} & \bar{G}_{1}&\ldots & \bar{G}_{N-1}\\
			\bar{G}_{1} &  \bar{G}_{2} &\iddots  & \bar{G}_{N} \\
			 \vdots & \iddots  &\iddots& \vdots \\
			 \bar{G}_{N-1} & \bar{G}_{N} &\ldots & \bar{G}_{2N-2}
	\end{array}\right) 
\end{displaymath}
and
\begin{displaymath}
	\bar{\mathbf{H}}_{1,N}=
	\left( \begin{array}{cccc}
			\bar{G}_{1} & \bar{G}_{2}&\cdots & \bar{G}_{N}\\
			\bar{G}_{2} & \bar{G}_{3}  &\iddots  & \bar{G}_{N+1} \\
			 \vdots & \iddots  & \iddots  & \vdots \\
			 \bar{G}_{N}  & \bar{G}_{N+1}  &\ldots& \bar{G}_{2N-1}
	\end{array}\right).
\end{displaymath}
Since an arbitrary Hankel matrix of finite rank admits a Vandermonde decomposition, these matrices can be factorized as \cite{boley97}:
\begin{equation}
\label{eq:jun:22:17:30}
\bar{\mathbf{H}}_{0,N}=\mathbf{V} \mathbf{D}_{0} \mathbf{V}^{T},
\end{equation}
and
\begin{equation}
\label{eq:jun:22:17:31}
\bar{\mathbf{H}}_{1,N}=\mathbf{V} \mathbf{D}_{1} \mathbf{V}^{T},
\end{equation}
where
\begin{displaymath}
	\mathbf{V}=
	\left( \begin{array}{cccc}
			1 & 1&\ldots & 1\\
			\zeta_{1}^{\Delta\gamma} &\zeta_{2}^{\Delta\gamma} &\ldots  &  \zeta_{N}^{\Delta\gamma}\\
			 \vdots & \vdots  &\vdots& \vdots \\
			\zeta_{1}^{(N-1)\Delta\gamma} &\zeta_{2}^{(N-1)\Delta\gamma} &\ldots  &  \zeta_{N}^{(N-1)\Delta\gamma}\\
	\end{array}\right),
\end{displaymath}
is a Vandermonde matrix,
\begin{eqnarray}
\label{eq:jan:26:12:38}
\mathbf{D}_{0}=diag(\zeta_{1}^{\gamma_{0}}n_{1},\zeta_{2}^{\gamma_{0}}n_{2},\cdots, \zeta_{N}^{\gamma_{0}}n_{N}),
\end{eqnarray}
and
\begin{eqnarray}
\label{eq:jan:26:12:39}
\mathbf{D}_{1}=diag(\zeta_{1}^{\gamma_{0}+\Delta\gamma}n_{1}, \zeta_{2}^{\gamma_{0}+\Delta\gamma}n_{2},\cdots,\zeta_{N}^{\gamma_{0}+\Delta\gamma}n_{N}).
\end{eqnarray}
From these expressions, one can easily prove the following two theorems \cite{kravanja98}:
\begin{theorem}
\label{theorem:may:04:11:32}
Let $N$ be the number of zeros and poles, then $N=rank(\bar{\mathbf{H}}_{0,N+p})$ for every $p\in\mathbb{N}$.
\end{theorem}
Theorem \ref{theorem:may:04:11:32}  gives us the number of the zeros and poles $N$.
\begin{theorem}
\label{theorem:may:04:16:45}
The eigenvalues of the generalized eigenvalue problem 
\begin{eqnarray}
\label{eq:jan:26:15:59}
\bar{\mathbf{H}}_{1,N}\vec{x}=\lambda\bar{\mathbf{H}}_{0,N}\vec{x}
\end{eqnarray}
are given by $\lambda_{i}=\zeta_{i}^{\Delta\gamma}$,  with the corresponding eigenvectors $\vec{x}_{i}=\mathbf{V}^{-T}\hat{e}_{i}$, $\hat{e}_{i}$ is the unit vector in $i$-direction.
\end{theorem}
It is worthwhile mentioning that Theorem \ref{theorem:may:04:16:45} is essentially a reformulated Prony's method, which has been widely used in spectral evaluation \cite{hildebrand87}.
The generalized eigenvalue problem,  Eq.(\ref{eq:jan:26:15:59}), can be solved by a QZ algorithm with $\mathcal{O}(N^{3})$ operations \cite{golub132}, yielding  the desired zeros and poles $\zeta_{j}$ in $\zeta$-space.
Once $\zeta_{j}$ are known, the multiplicities can be obtained through a Vandermonde system (Eq.(\ref{eq:jan:25:16:01})). As the multiplicities must be integers, this step is relatively robust.
Mathematically, by using the coordinate transformation $z=T^{-1}(\zeta)$,  one can thus recover the locations of zeros and poles of $f(z)$ within $\mathbb{D}$. 

\section{Sensitivity Analysis}
\label{sec:feb:09:15:11}
In this section, the parameters $\gamma_{0}, \Delta\gamma$ and transformation $\zeta=T(z)$ are determined via the sensitivity analysis.
Noting that, since both Vandermonde and Hankel matrices can be ill-conditioned, it's impossible to assert the algorithm stability in any universal sense. 
However, one can turn to the question of what actually affects the sensitivity, and how to find a stable parameter regime in practical applications. 

For simplicity and hence clarity,  we assume, without loss of generality, that the complex domain $\mathbb{D}_{\zeta}$ is bounded by a circle $\mathcal{C}$ with the center $\zeta=0$ and radius $r_{c}$, and there are $N$ zeros and poles $\zeta_{j}$ in it.  We employ the standard perturbative treatment of the generalized eigenvalue problem \cite{golub131}. That is,  Eq.(\ref{eq:jan:26:15:59}) is perturbed  as
\begin{equation}
\label{eq:jun:14:22:58}
[(\bar{\mathbf{H}}_{1,N}+\epsilon\hat{\mathbf{H}}_{1,N})-(\lambda+\epsilon\hat{\lambda})(\bar{\mathbf{H}}_{0,N}+\epsilon\hat{\mathbf{H}}_{0,N})](\vec{x}+\epsilon\hat{x})=0,
\end{equation}
where $\epsilon$ is a small expansion parameter,  $\hat{\mathbf{H}}_{0,N},\hat{\mathbf{H}}_{1,N}$ are the normalized perturbations with Hankel structure.
To the first order in $\epsilon$, we have
\begin{equation}
\label{eq:jun:14:23:02}
(\bar{\mathbf{H}}_{1,N}-\lambda\bar{\mathbf{H}}_{0,N})\hat{x}=(\lambda\hat{\mathbf{H}}_{0, N}+\hat{\lambda}\bar{\mathbf{H}}_{0, N}-\hat{\mathbf{H}}_{1, N})\vec{x}.
\end{equation}
Given an arbitrary eigenvalue $\lambda_{i}$ and the corresponding eigenvector $\vec{x}_{i}$, the symmetry in Hankel matrices then gives
\begin{equation}
\label{eq:jun:15:16:54}
\vec{x}^{T}_{i}(\bar{\mathbf{H}}_{1, N}-\lambda_{i}\bar{\mathbf{H}}_{0, N})=0.
\end{equation}
Pre-multiplying Eq.(\ref{eq:jun:14:23:02}) by $\vec{x}_{i}^{T}$ and taking the norm, we obtain
\begin{equation}
\label{eq:jun:15:18:09}
|\hat{\lambda}_{i}|=\frac{|\vec{x}^{T}_{i}\hat{\mathbf{H}}_{1}\vec{x}_{i}|-|\lambda_{i} \vec{x}^{T}_{i}\hat{\mathbf{H}}_{0}\vec{x}_{i}|}{|\vec{x}^{T}_{i}\bar{\mathbf{H}}_{0}\vec{x}_{i}|},
\end{equation}
and $\epsilon\hat{\lambda}_{i}$ gives the error estimation.
Noting $\vec{x}_{i}=\mathbf{V}^{-T}\hat{e}_{i}$, we find
\begin{equation}
\label{eq:jun:22:16:40}
|\vec{x}^{T}_{i}\bar{\mathbf{H}}_{0, N}\vec{x}_{i}|=|\vec{x}^{T}_{i} \mathbf{V} \mathbf{D_{0}} \mathbf{V}^{T} \vec{x}_{i}|=|\zeta_{i}^{\gamma_{0}}n_{i}|,
\end{equation}
and the corresponding sensitivity can thus be estimated as
\begin{eqnarray}
\label{eq:jun:22:15:52}
|\hat{\lambda}_{i}|\le\|\vec{e}^{T}_{i}\mathbf{V}^{-1}\|_{\infty}^{2}\frac{\|\hat{\mathbf{H}}_{1, N}\|_{\infty}+|\zeta_{i}^{\Delta\gamma} |\|\hat{\mathbf{H}}_{0, N}\|_{\infty}}{|\zeta_{i}^{\gamma_{0}}n_{i}|}.
\end{eqnarray}
Postulating further $\|\hat{\mathbf{H}}_{1, N}\|_{\infty}=\|\bar{\mathbf{H}}_{1}\|_{\infty}$, $\|\hat{\mathbf{H}}_{0, N}\|_{\infty}=\|\bar{\mathbf{H}}_{0}\|_{\infty}$, then,  by noticing that 
\begin{equation}
\label{eq:jun:22:17:23}
\|\mathbf{V}\|_{\infty}\le N \max_{k,j=1}^{N}\{|\zeta_{k}^{(j-1)\Delta\gamma}|\},
\end{equation}
\begin{equation}
\label{eq:jun:22:17:25}
\|\bar{\mathbf{H}}_{0,N}\|_{\infty}\le \|\mathbf{V}\|^{2}_{\infty}\max_{j=1}^{N}\{|n_{j}\zeta_{j}^{\gamma_{0}}|\},
\end{equation}
\begin{equation}
\label{eq:jun:22:17:55}
\|\bar{\mathbf{H}}_{1, N}\|_{\infty}\le\|\mathbf{V}\|^{2}_{\infty}\max_{j=1}^{N}\{|n_{j}\zeta_{j}^{\gamma_{0}+\Delta\gamma}|\},
\end{equation}
and the property of inverse Vandermonde matrix  \cite{cit:feb:03:10:43},
\begin{equation}
\label{eq:jun:22:17:12}
\|\vec{e}^{T}_{i}\mathbf{V}^{-1}\|_{\infty}\le\prod_{j\ne i}\frac{1+|\zeta_{j}^{\Delta\gamma}|}{|\zeta_{i}^{\Delta\gamma}-\zeta_{j}^{\Delta\gamma}|},
\end{equation}
it is possible to show that a smaller $\gamma_{0}$ can improve the condition of the system. Correspondingly, we should adopt $\gamma_{0}=0$. 
In this way, Eq.(\ref{eq:jun:22:15:52}) can be rendered into
\begin{eqnarray}
\label{eq:jun:22:20:29}
|\hat{\lambda}_{i}|\le \kappa_{\infty}^{2}(\max_{j=1}^{N}\{|\zeta_{j}^{\Delta\gamma}|\}+|\zeta_{i}^{\Delta\gamma} |),
\end{eqnarray}
with the condition number $\kappa_{\infty}^{2}=\|\vec{e}^{T}_{i}\mathbf{V}^{-1}\|^{2}_{\infty}\|\mathbf{V}\|^{2}_{\infty}$.
Letting
\begin{equation}
\label{eq:feb:03:15:21}
r_{+}=\max_{j=1}^{N}(|\zeta_{j}|),\quad r_{-}=\min_{j=1}^{N}(|\zeta_{j}|),
\end{equation}
Eqs. (\ref{eq:jun:22:17:23}) and (\ref{eq:jun:22:17:12}) can be cast, respectively, as
\begin{equation}
\label{eq:jun:22:22:54}
\|\mathbf{V}\|_{\infty}\le N r_{+}^{(N-1)\Delta\gamma}, \quad \textrm{with $r_{+}\ge 1$,}
\end{equation}
and
\begin{equation}
\label{eq:jun:22:23:02}
\|\vec{e}^{T}_{i}\mathbf{V}^{-1}\|_{\infty}\le r_{-}^{-(N-1)\Delta\gamma}\prod_{j\ne i}\frac{1+|\zeta_{j}^{\Delta\gamma}|}{|e^{i\theta_{i}\Delta\gamma}-e^{i\theta_{j}\Delta\gamma}|},
\end{equation}
with  $\zeta_{j}=|\zeta_{j}|\textrm{Exp}(i\theta_{j})$.
Therefore, the condition number can be explicitly written as
\begin{eqnarray}
\label{eq:jun:22:23:08}
\kappa_{\infty}^{2}&\le& N^{2} r_{+}^{2(N-1)\Delta\gamma}[1+r_{+}^{\Delta\gamma}]^{2(N-1)}(\prod_{j\ne i}\frac{1}{|\zeta_{i}^{\Delta\gamma}-\zeta_{j}^{\Delta\gamma}|})^{2}\nonumber\\
&\le& N^{2} (\frac{r_{+}}{r_{-}})^{2(N-1)\Delta\gamma}[1+r_{+}^{\Delta\gamma}]^{2(N-1)}(\prod_{j\ne i}\frac{1}{|e^{i\theta_{i}\Delta\gamma}-e^{i\theta_{j}\Delta\gamma}|})^{2}.
\end{eqnarray}
From Eq. (\ref{eq:jun:22:23:08}), it is clear that the system is always stable for the $N=1$ case, as expected;
 and the system is more stable for small $\Delta\gamma$ and $r_{+}$. Thus we should set $\Delta \gamma=1$ and $r_{+}=1$. 
In addition, Eq.(\ref{eq:jun:22:23:08}) also demonstrates that the sensitivity is determined by the number and locations of zeros and poles, regardless of the $f(\zeta)$ values along the Jordan curve. In particular, the large $N$ and $r_{+}/r_{-}$, and the existence of clusters can make the system very ill-conditioned.
For zeros and poles distributed uniformly around the unit circle, however, the system is stable even with a large $N$ number. 

This fact allows us to adopt a subdivision-transformation-calculation scheme to ensure the algorithm stability.
More specifically, the subdivision process is carried out in $z$-space. Noting that the subdivision scheme has been studied extensively by previous works \cite{delves67,kravanja98,kravanja99,kravanja99b}, it is not necessary to go into details here.  In this study, for simplicity, we search in rectangles to avoid redundancy during subdivision, with a generic example of the rectangle region searched shown in Fig. (\ref{eps:rectangle}).
Once the rectangular vertexes, namely the points A, B, C and D in Fig. (\ref{eps:rectangle}), are given, the rectangle in $z$-space can be transformed to a slotted annulus right next to the unit circle in $\zeta$-space, as seen in Fig. (\ref{eps:circle}), via the transformation:
\begin{eqnarray}
\label{eq:tranform}
	\zeta\equiv T(z)=e^{-\frac{i(2\pi-\epsilon_{0})(z e^{-i\alpha}-z_{0})}{L}}.
\end{eqnarray}
Here, the small positive parameter $\epsilon_{0}\ll 1$ is introduced to keep away from the branch cut of the complex logarithm in $\zeta$-space, $\alpha\in (-\pi,\pi]$ is the angle between the positive $\textrm{Re}(z)$-axis and the line $\overline{AB}$, and $z_{0}$ and $L$ are, respectively, the midpoint and length of $\overline{CD}$.
Then, zeros and poles $(\zeta_{i}, n_{i})$  are preliminarily calculated in $\zeta$-space by using the algorithm presented in Sec. (\ref{sec:feb:09:15:10}), and the associated condition number of the Prony system is estimated by Eq. (\ref{eq:jun:22:23:08}).
If the condition number is larger than the preassigned value, one should suitably subdivide the region into smaller subregions in $z$-space,  transform each subregion into $\zeta$-space, and calculate the corresponding zeros and poles. 
The subdivision-transformation-calculation process is repeated until the resulting condition number is acceptable.

Furthermore, it is also worthwhile noting that one feasibility of Eq. (\ref{eq:jun:22:23:08}) lies in the identifying of primary factors influencing the stability of the system.
Consistently with the fact that the total time taken is mainly set by the number of regions, it is possible and desirable to develop a  more sophisticated subdivision scheme and, thus, to significantly improve the efficiency of algorithm.
This topic will be pursued in future publications.


\begin{figure}[!htp]
\centering
	\subfloat[]{
\label{eps:rectangle}
\begin{minipage}[t]{0.45\textwidth}
\centering
\includegraphics[scale=0.30]{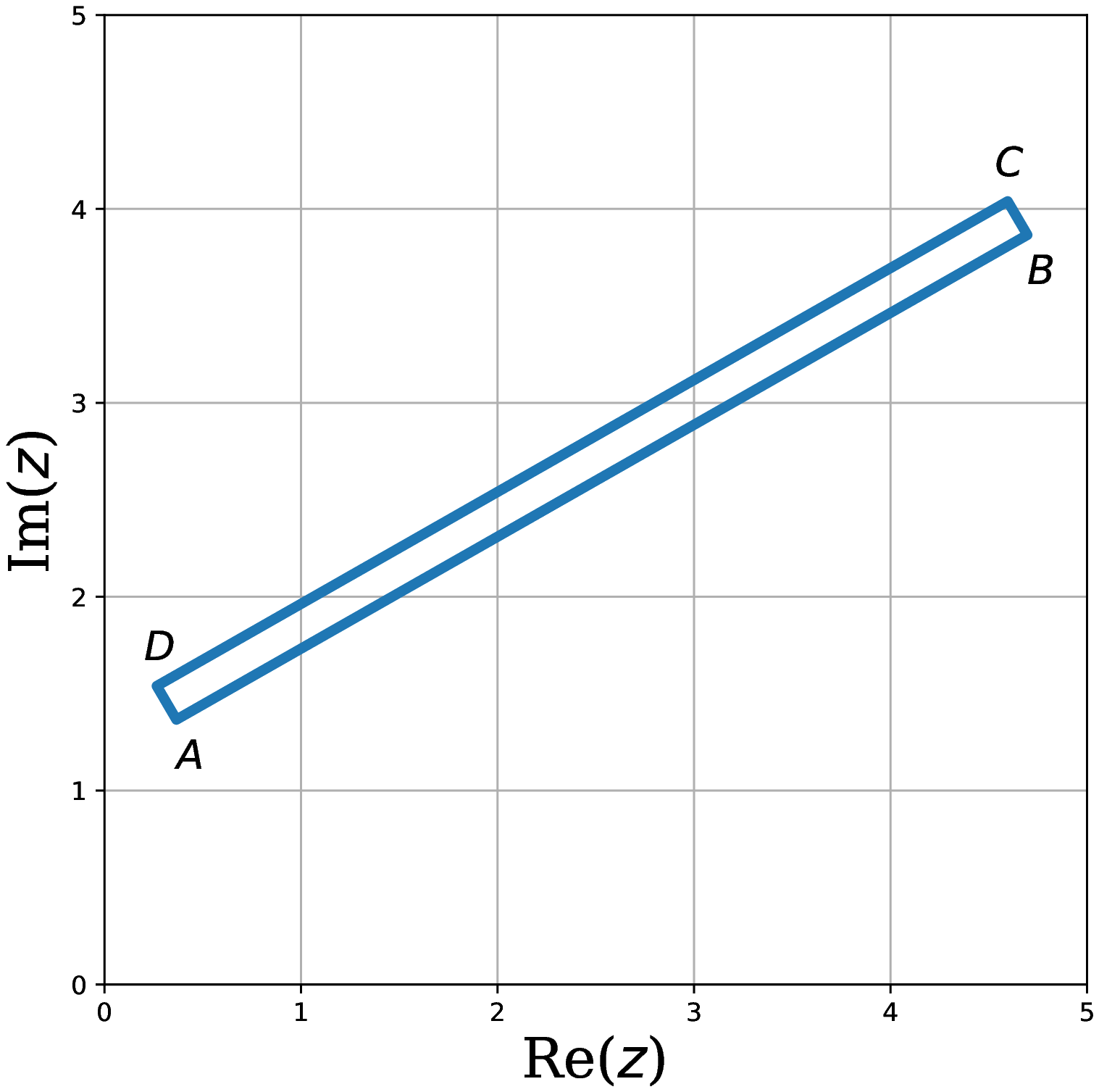}
\end{minipage}
}
	\subfloat[]{
\label{eps:circle}
\begin{minipage}[t]{0.45\textwidth}
\centering
\includegraphics[scale=0.30]{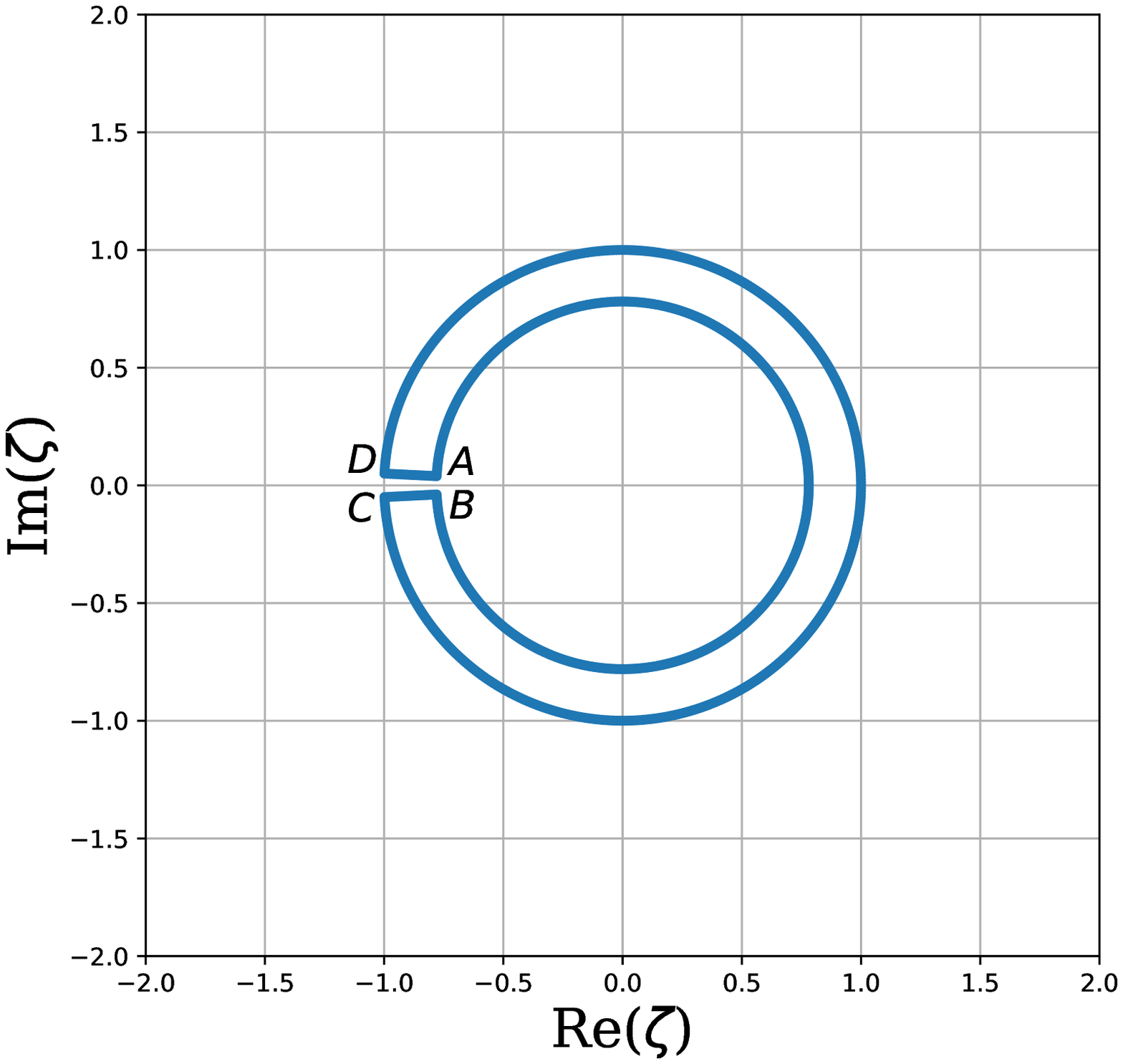}
\end{minipage}
}
\caption{(Color online) A rectangle in $z$-space is transformed to a slotted annulus in $\zeta$-space. }
\label{eps:transform}
\end{figure}

\section{Error Analysis}
\label{sec:feb:09:15:12}
As the generalized eigenvalue problem can be solved in well-condition, the algorithm still suffers the inherent numerical error stemming from the contour integration in $\zeta$-space. 

Since the numerical differentiation is time-consuming and error-prone, we perform an integration by parts in Eq.(\ref{eq:jan:25:16:01}) and use the logarithmic derivative to avoid evaluating $f'$. In this case,  the main difficulty is the multivalueness of the complex logarithm. Thus,  in order to identify the same branch of $\ln f$ numerically, we write
\begin{equation}
\label{eq:jul:28:14:51}
\ln f(\zeta)=\ln|f|(\zeta)+i\Theta(\zeta),
\end{equation}
with
\begin{equation}
\label{eq:feb:05:12:11}
\Theta=Arg(f)+2\pi m, \textrm{with $m\in\mathbb{Z}, \Theta\in \mathbb{R}$,}
\end{equation}
and  ensure a continuous extended argument $\Theta$ by keeping track of the numerical calculated principal value $Arg(f)$ and selecting appropriate $m$s. 
Then Eq.(\ref{eq:jan:25:16:01}) becomes
\begin{equation}
\label{eq:feb:05:13:14}
\bar{G}_{k}=\frac{1}{2\pi i}\{[\zeta^{k} \ln f ]_{\zeta_{s}}^{\zeta_{e}}-\ointctrclockwise_{\mathcal{C}} k \zeta^{k-1}\ln f d \zeta\},
\end{equation}
with $\zeta_{s} $  and $\zeta_{e}$ denoting the starting-point and end-point of the contour integration.
Introducing the complex logarithm in this way provides several benefits. First, it avoids the evaluating of $f'$. Second, since $[\zeta^{k} \ln f ]_{\zeta_{s}}^{\zeta_{e}}/(2\pi i)$ is an integer, it can be computed accurately as with the Nyquist stability criterion. Third, the complex logarithm is more robust against the overflow and underflow issues.

Having obtained the zeros and poles, an important question arises as to how to estimate the error level. This can be addressed by the following theorem.
\begin{theorem}
\label{theorem:aug:23:10:42}
For small $p\in\mathbb{N}^{+}$, the eigenvalues of the singular pencil $\bar{\mathbf{H}}_{1,N+p}-\lambda\bar{\mathbf{H}}_{0,N+p}$, which is corrupted by noise, fall into two categories:
\begin{enumerate}
	\item $\{ \zeta_{1},\zeta_{2},\cdots,\zeta_{N}\}$, i.e., the eigenvalues of the pencil $\bar{\mathbf{H}}_{1,N}-\lambda\bar{\mathbf{H}}_{0,N}$, which are independent of $p$;
	\item $\{\eta_{1},\eta_{2},\cdots,\eta_{p}\}$, which depend on $p$. 
\end{enumerate}
\end{theorem}
\begin{proof}
	First, we redefine Eq.(\ref{eq:jan:25:12:59})  as
\begin{equation}
\label{eq:aug:23:11:15}
g^{(m)}_{k}(\zeta) =\zeta^{k+m}, m\in \mathbb{N},
\end{equation}
and denote the $(i,j)$ minor matrices of $\bar{\mathbf{H}}^{(m)}_{0,N+1}, \bar{\mathbf{H}}^{(m)}_{1,N+1}$ as $\mathbf{M}^{(m)}_{0,ij}$ and $\mathbf{M}^{(m)}_{1,ij}$, respectively. 

Then the minors can be decomposed as 
\begin{equation}
\label{eq:feb:06:17:09}
\mathbf{M}_{0,ij}^{(m)}=\mathbf{V}_{i} \mathbf{D}_{0} \mathbf{V}^{T}_{j},
\end{equation}
and
\begin{equation}
\label{eq:feb:06:17:10}
\mathbf{M}_{1,ij}^{(m)}=\mathbf{V}_{i} \mathbf{D}_{1} \mathbf{V}^{T}_{j},
\end{equation}
where
\begin{eqnarray}
\label{eq:fen:06:17:57}
\mathbf{D}_{0}=diag(\zeta_{1}^{m}n_{1},\zeta_{2}^{m}n_{2},\cdots, \zeta_{N}^{m}n_{N}),
\end{eqnarray}
\begin{eqnarray}
\label{eq:fen:06:17:58}
\mathbf{D}_{1}=diag(\zeta_{1}^{m+1}n_{1}, \zeta_{2}^{m+1}n_{2},\cdots,\zeta_{N}^{m+1}n_{N}),
\end{eqnarray}
and
\begin{displaymath}
	\mathbf{V}_{i}=
	\left( \begin{array}{ccccccc}
			1 & 1&\ldots &\ldots &\ldots &\ldots & 1\\
			\zeta_{1} &\zeta_{2}&\ldots&\ldots&\ldots&\ldots  &  \zeta_{N}\\
			 \vdots & \vdots & \vdots & \vdots & \vdots  &\vdots& \vdots \\
			\zeta_{1}^{i-2} &\zeta^{i-2}_{2}&\ldots&\ldots&\ldots&\ldots  &  \zeta^{i-2}_{N}\\
			\zeta_{1}^{i} &\zeta^{i}_{2}&\ldots&\ldots&\ldots&\ldots  &  \zeta^{i}_{N}\\
			 \vdots & \vdots & \vdots & \vdots & \vdots  &\vdots& \vdots \\
			\zeta_{1}^{N} &\zeta^{N}_{2}&\ldots&\ldots&\ldots&\ldots  &  \zeta^{N}_{N}\\
	\end{array}\right),
\end{displaymath}
which is just the result of removing $i$-th row from a $(N+1)\times N$ Vandermonde matrix $\mathbf{V}$. Similarly, we can get $\mathbf{V}^{T}_{j}$ by removing $j$-th column from a $N\times (N+1)$ matrix $\mathbf{V}^{T}$.
Therefore,   a new generalized eigenvalue problem $\mathbf{M}^{(m)}_{1,ij}\vec{x}=\lambda\mathbf{M}^{(m)}_{0,ij}\vec{x}$ can be constructed. 

	Considering that
\begin{eqnarray}
\label{eq:fen:06:18:25}
& &\mathbf{M}^{(m)}_{1,ij}\vec{x}=\lambda\mathbf{M}^{(m)}_{0,ij}\vec{x}\\
&\Leftrightarrow&\mathbf{V}_{i} \mathbf{D}_{1} \mathbf{V}^{T}_{j}\vec{x}=\lambda\mathbf{V}_{i} \mathbf{D}_{0} \mathbf{V}^{T}_{j}\vec{x}\nonumber\\
&\Leftrightarrow& diag(\zeta_{1},\zeta_{2},\cdots,\zeta_{N})( \mathbf{V}^{T}_{j}\vec{x})=\lambda(\mathbf{V}^{T}_{j}\vec{x}),\nonumber
\end{eqnarray}
one readily finds that the pencil $\mathbf{M}^{(m)}_{1,ij}-\lambda\mathbf{M}^{(m)}_{0,ij}$ have the same eigenvalues with $\bar{\mathbf{H}}_{1,N}-\lambda\bar{\mathbf{H}}_{0,N}$, namely, $\{\zeta_{i}\}$ are also solutions to $\det(\mathbf{M}^{(m)}_{1,ij}-\lambda\mathbf{M}^{(m)}_{0,ij})\equiv M=0$. 
Thus, a cofactor expansion in row yields
\begin{equation}
\label{eq:aug:23:12:20}
	\det(\bar{\mathbf{H}}_{1,N+p}-\lambda\bar{\mathbf{H}}_{0,N+p})=\sum_{l}a_{l}M_{l},
\end{equation}
	where $l$ represents the number of $N\times N$ minors, and  $M_{l}$ denotes the associated determinants.
	Since the formal solution of each $M_{l}=0$ is given by $\{\zeta_{i}\}$, it follows that $\{\zeta_{i}\}$ are also eigenvalues of the corrupted pencil $\bar{\mathbf{H}}_{1,N+p}-\lambda\bar{\mathbf{H}}_{0,N+p}$. 
	Meanwhile, the argument for sensitivity analysis in the previous section  can be straightforwardly repeated. As a result,  one obtains that the condition number of Eq. (\ref{eq:fen:06:18:25}) is $(r_{+}/r_{-})^{m}$ higher than that in Eq. (\ref{eq:jun:22:23:08}), and for a small $p$ with $(r_{+}/r_{-})^{p}\sim \mathcal{O}(1)$, the condition number of the corrupted pencil is of the same order as that of the original pencil $\bar{\mathbf{H}}_{1,N}-\lambda\bar{\mathbf{H}}_{0,N}$.
	 Furthermore, recalling that  the QZ-algorithm does not involve rank determination and matrix inversion, indeed we can obtain $\{\zeta_{i}\}$ from the corrupted pencil  $\bar{\mathbf{H}}_{1,N+p}-\lambda\bar{\mathbf{H}}_{0,N+p}$ via a QZ-algorithm in well-condition \cite{golub132}.

This proves the theorem.
\end{proof}
Theorem (\ref{theorem:aug:23:10:42}) provides a convenient way to estimate the order of absolute errors. Let $\zeta_{i}^{(0)}$ and $\zeta_{i}^{(1)}$ be the same eigenvalue calculated, respectively,  from $\bar{\mathbf{H}}_{1,N}-\lambda\bar{\mathbf{H}}_{0,N}$ and $\bar{\mathbf{H}}_{1,N+1}-\lambda\bar{\mathbf{H}}_{0,N+1}$, then the numerical error can be simply estimated by
\begin{eqnarray}
\label{eq:aug:23:12:35}
\delta_{i}=\frac{|\zeta_{i}^{(0)}-\zeta_{i}^{(1)}|}{2}.
\end{eqnarray}
It should be emphasized that, due to the $N+1$ cofactors involved  and $|a_{l}|\sim \mathcal{O}(1)$ near the unit cile, $\zeta_{i}^{(1)}$ is essentially a weighted average over the results of $N+1$ equivalent pencils.

\section{Numerical examples}
\label{sec:feb:09:15:13}

To verify the strategies described in Secs. (\ref{sec:feb:09:15:11}) and (\ref{sec:feb:09:15:12}), we have implemented the algorithm to carry out illustrative examples.
Specifically, the algorithm can be sketched as follows:
\begin{enumerate}[Step 1]
	\item Set a critical condition number $\kappa^{2}_{c}$ and an error tolerance of the contour integration $\epsilon_{i}$;
	\item    Transform the rectangle region searched in $z$-space, say $\mathbb{D}$, to a slotted annulus in $\zeta$-space, with $\epsilon_{0}$ introduced to avoid the branch cut of the complex logarithm;
	\item  Construct Hankel matrices, determine their ranks and solve the pencil $\bar{\mathbf{H}}_{1,N}-\lambda\bar{\mathbf{H}}_{0,N}$;
	\item Calculate the condition number $\kappa^{2}_{\infty}$ using Eq.(\ref{eq:jun:22:23:08}),
		\begin{itemize}
			\item if $\kappa^{2}_{\infty}>\kappa^{2}_{c}$, subdivide $\mathbb{D}$ into smaller subregions $\mathbb{D}_{i}$, and go back to Step 2; 
			\item otherwise, continue;
		\end{itemize}
	\item Solve the corrupted pencil $\bar{\mathbf{H}}_{1,N+1}-\lambda\bar{\mathbf{H}}_{0,N+1}$ for error estimate;
	\item Multiplicities are obtained via the associated Vandermonde system, i.e., Eq.(\ref{eq:jan:25:16:01}).
\end{enumerate}
The following numerical examples serve as tests of the algorithm, 
and illuminate the potential applications  to plasma physics.
Multiplicities are recovered successfully in all cases.  


\subsubsection{Example 1:}
Considering a trivial test case
\begin{eqnarray}
\label{eq:feb:07:15:51}
f=\frac{(z-0.8-0.9i)(z-0.7+0.8i)(z+0.6+0.7i) }{(z+0.5-0.6i)^{2}},
\end{eqnarray}
and taking $\kappa_{c}^{2}=128, \epsilon_{i}=1.49\times 10^{-8}$ and $\epsilon_{0}=0.1$, the numerically computed zeros and poles of $f$ are shown in Tab.(\ref{tab:apr:14:14:12}). Here, $\delta_{i,e}$ and $\delta_{i,t}$ stand for, respectively, the error estimate given by Theorem (\ref{theorem:aug:23:10:42}) and the true error. Table (\ref{tab:apr:14:14:12}) demonstrates that the numerical error estimate is reasonably accurate.
\begin{table}[!htp]
\resizebox{\textwidth}{!}{
\begin{tabular}{ccc}
\hline\noalign{\smallskip}
$z_{i}$& $\delta_{i,e}$& $\delta_{i,t}$\\
\noalign{\smallskip}\hline\noalign{\smallskip}
$-0.5999999999753678-0.6999999999322971i$ & $1.36\times 10^{-10}$ & $ 7.20\times 10^{-11}$\\
$0.7000000004745937-0.7999999997652205i$ & $3.91\times 10^{-10}$ & $5.29\times 10^{-10} $\\
$0.7999999995583811+0.9000000002491819i$ & $3.55\times 10^{-10}$ & $5.07\times 10^{-10} $\\
$-0.5000000000007568+0.5999999999992878i$ & $3.12\times 10^{-12}$ & $1.04\times 10^{-12} $\\
\noalign{\smallskip}\hline
\end{tabular}
}
\caption{Numerically computed zeros and poles.}
\label{tab:apr:14:14:12}
\end{table}

\subsubsection{Example 2:}
An efficient method for the nonlinear eigenvalue problem has not been found. 
From Secs. (\ref{sec:feb:09:15:11}) and (\ref{sec:feb:09:15:12}), it is evident that the present algorithm is stable and easy-to-parallel.
Therefore, with the advance of the fast-increasing computational power, it may offer a possible approach to analysing the nonlinear eigenvalue problem.
As an example, we solve a transcendental eigenvalue problem \cite{cit:feb:08:13:49}:
$\det((e^{\lambda}-1)\mathbf{A}_{2}+\lambda^{2}\mathbf{A}_{1}-\mathbf{A}_{0})=0$,
with
\begin{displaymath}
	\mathbf{A}_{2}=
	\left( \begin{array}{ccc}
			17.6 & 1.28& 2.89\\
			1.28 & 0.824& 0.413\\
			2.89 & 0.413& 0.725
	\end{array}\right),
	\mathbf{A}_{1}=
	\left( \begin{array}{ccc}
			7.66 & 2.45& 2.1\\
			0.23 & 1.04& 0.223\\
			0.6 & 0.756& 0.658 
	\end{array}\right),
\end{displaymath}
\begin{displaymath}
	\mathbf{A}_{0}=
	\left( \begin{array}{ccc}
			12.1 & 18.9& 15.9\\
			0 & 2.7& 0.145\\
			11.9 & 3.64& 15.5 
	\end{array}\right).
\end{displaymath}
Shown in Table.(\ref{tab:feb:08:15:48}) are the simple zeros within the region $|Im(z_{i})|, |Re(z_{i})|\le 10$ and their associated errors. Again, the errors are estimated reasonably accurate.

\begin{table}[!htp]
\resizebox{\textwidth}{!}{
\begin{tabular}{ccc}
\hline\noalign{\smallskip}
$z_{i}$& $\delta_{i,e}$& $\delta_{i,t}$\\
\noalign{\smallskip}\hline\noalign{\smallskip}
$0.065949131387977-1.10\times 10^{-12}i$ & $2.70\times 10^{-12}$ & $ 1.33\times 10^{-12}$\\
$0.853377172251995+2.12\times 10^{-12}i$ & $5.39\times 10^{-12}$ & $ 2.48\times 10^{-12}$\\
$3.638975634806435+3.22\times 10^{-11}i$ & $6.81\times 10^{-11}$ & $ 3.59\times 10^{-11}$\\
$-5.587398329471895+9.17\times 10^{-13}i$ & $1.62\times 10^{-14}$ & $ 9.17\times 10^{-13}$\\
$-1.940259421974321+4.22\times 10^{-12}i$ & $7.76\times 10^{-12}$ & $ 4.61\times 10^{-12}$\\
$-0.936953776134564+4.39\times 10^{-12}i$ & $7.02\times 10^{-12}$ & $ 4.43\times 10^{-12}$\\
$4.750269139855016-5.443800760044676i$ & $1.74\times 10^{-14}$ & $ 2.24\times 10^{-13}$\\
$3.061926419734661-5.265134384625599i$ & $1.06\times 10^{-12}$ & $ 4.38\times 10^{-12}$\\
$3.858870604351882-4.985782136928656i$ & $1.83\times 10^{-13}$ & $ 4.00\times 10^{-12}$\\
$3.858870604352364+4.985782136922126i$ & $9.89\times 10^{-13}$ & $ 7.21\times 10^{-12}$\\
$3.061926419737111+5.265134384628629i$ & $6.44\times 10^{-12}$ & $ 3.17\times 10^{-12}$\\
$4.750269139854812+5.443800760044741i$ & $2.07\times 10^{-14}$ & $ 1.16\times 10^{-13}$\\
\noalign{\smallskip}\hline
\end{tabular}
}
	\caption{Numerically computed zeros of $\det((e^{\lambda}-1)\mathbf{A}_{2}+\lambda^{2}\mathbf{A}_{1}-\mathbf{A}_{0})=0$. $\kappa^{2}_{c}, \epsilon_{i}$ and $\epsilon_{0}$ are the same as Table (\ref{tab:apr:14:14:12})}
\label{tab:feb:08:15:48}
\end{table}

\subsubsection{Example 3:}
In this example, we consider the zeros of plasma dispersion function \cite{cit:feb:27:11:02}, which is widely used to model the wave-particle interaction in kinetic plasma turbulence. Specifically, the function is defined as 
\begin{eqnarray}
\label{eq:feb:07:18:32}
Z(z)=\frac{1}{\sqrt{\pi}}\int_{-\infty}^{+\infty}dv\frac{e^{-v^{2}}}{v-z},\quad \textrm{with Im$(z)>1$,}
\end{eqnarray}
and as its analytic continuation for $\textrm{Im}(z)\le 0$. Due to the symmetry property $Z(z^{*})=-[Z(-z)]^{*}$, it is straightforward to show that zeros of the plasma dispersion function occur in real conjugate  pairs.
By taking the same $\kappa^{2}_{c}, \epsilon_{i}$ and $\epsilon_{0}$ as before,  the numerical computed simple zeros within the range $\textrm{Im}(z)\ge-5$ are listed in Table.(\ref{tab:feb:07:19:20}). Furthermore, it is interesting to note that, 
zeros of the plasma dispersion function approach $\textrm{Im}(z)=-|\textrm{Re}(z)|$ in the $|z|\gg 1$ limit  (as seen in Fig. (\ref{eps:feb:07:19:35})), which is a typical feature of the error function. In fact, an alternative representation of $Z(z)$ is
\begin{eqnarray}
\label{eq:feb:07:19:50}
	Z(z)=i\sqrt{\pi}e^{-z^{2}}[1+\textrm{erf}(iz)].
\end{eqnarray}
\begin{table}[!htp]
\resizebox{\textwidth}{!}{
\begin{tabular}{ccc}
\hline\noalign{\smallskip}
$z_{i}$& $\delta_{i,e}$& $\delta_{i,t}$\\
\noalign{\smallskip}\hline\noalign{\smallskip}
$1.99146684283858-1.35481012808997i$ & $2.17\times 10^{-11}$ & $ 2.25\times 10^{-11}$\\
$2.69114902411825-2.17704490608676i$ & $1.40\times 10^{-11}$ & $ 1.33\times 10^{-11}$\\
$3.23533086843928-2.78438761010462i$ & $3.13\times 10^{-9}$ & $ 6.44\times 10^{-9}$\\
$3.69730970246813-3.28741078938962i$ & $1.22\times 10^{-14}$ & $ 4.14\times 10^{-13}$\\
$4.10610728467995-3.72594871944305i$ & $6.08\times 10^{-12}$ & $ 3.83\times 10^{-12}$\\
$4.47681569296707-4.11963522761023i$ & $2.60\times 10^{-13}$ & $ 1.57\times 10^{-12}$\\
$4.81848829189866-4.47983279758210i$ & $2.47\times 10^{-10}$ & $ 1.50\times 10^{-10}$\\
$5.13706727240611-4.81380668333976i$ & $2.73\times 10^{-10}$ & $ 1.73\times 10^{-9}$\\
\noalign{\smallskip}\hline
\end{tabular}
}
	\caption{Zeros of the plasma dispersion function $Z(z)$ in the region $\textrm{Im}(z)\ge -5$. $\kappa^{2}_{c}, \epsilon_{i}$ and $\epsilon_{0}$ are the same as Table (\ref{tab:apr:14:14:12})}
\label{tab:feb:07:19:20}
\end{table}
\begin{figure}[!htp]
\centering
\includegraphics[scale=0.45]{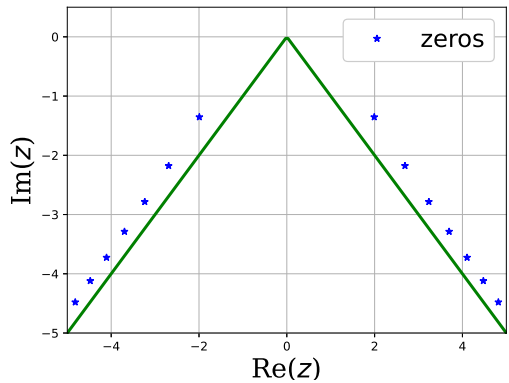}
	\caption{Zeros of the plasma dispersion function $Z(z)$ in the region $\textrm{Im}(z)\ge -5$. $\kappa^{2}_{c}, \epsilon_{i}$ and $\epsilon_{0}$ are the same as Table (\ref{tab:apr:14:14:12})}
\label{eps:feb:07:19:35}
\end{figure}

\subsubsection{Example 4:}
As an example of considerable practical importance, the present algorithm has been applied to the systematic numerical investigation of low-frequency electromagnetic waves in finite-$\beta$ anisotropic plasmas \cite{chen21}. Here, $\beta=8\pi P/B^{2}$ is the ratio between kinetic and magnetic energy densities.

Electromagnetic fluctuations  with frequencies much lower than the ion cyclotron frequency are prevalent in nature and laboratory plasmas.
Theoretically,
a self-consistent description of these fluctuations can be derived from the so-called gyrokinetic Maxwell equations \cite{frieman,chen91,brizard}. 
Specifically, the plasma response is described by the linear gyrokinetic equation \cite{frieman},  while the electromagnetic perturbations are characterized by three fluctuating scalar fields: the electrostatic potential $\delta\phi$, the scalar induced potential $\delta\psi$ accounting for the perpendicular magnetic field fluctuation, and the compressional magnetic fluctuation $\delta B_{\parallel}$.
 Within this approach, the governing equations for the time evolution of $(\delta\phi, \delta\psi, \delta B_{\parallel})$ are the quasineutrality condition, perpendicular component of Ampere's law, and gyrokinetic vorticity equation  \cite{chen91}.

For a uniform and finite-$\beta$ plasma immersed in a uniform background magnetic field $\bm{B}=B_{0}\bm{e}_{z}$,  the fluctuating variables can be decomposed into Fourier series,
\begin{eqnarray}
\label{eq:fourier}
	[\delta\phi,\delta\psi,\delta B_{\parallel}]=\sum_{\bf{k}}[\delta\phi_{k},\delta\psi_{k},\delta B_{\parallel,k}]e^{i(\bf{k}\cdot\bf{x}-\omega t)}.
\end{eqnarray}
Assuming  an anisotropic bi-Maxwellian equilibrium distribution,
\begin{eqnarray}
\label{eq:F0}
	F_{0s}=\frac{N_{s}}{\pi^{3/2} v^{2}_{ts\perp} v_{ts\parallel}}e^{-\frac{v_{\parallel}^{2}}{v_{ts\parallel}^{2}}-\frac{v_{\perp}^{2}}{v_{ts\perp}^{2}}},
\end{eqnarray}
where the subscript $s$ denotes the particle species,  $N_{s}$ is the unperturbed particle density, $v_{ts\perp(\parallel)}$ is the perpendicular (parallel) thermal velocity and $T_{s\perp(\parallel)}=m_{s}v_{ts\perp(\parallel)}^{2}/2$ is the corresponding temperature,
the linear gyrokinetic Maxwell equations can thus be rendered into a complicated nonlinear eigenvalue problem:
\begin{eqnarray}
\label{eq:eigenmode}
	Q_{1}\Phi_{\parallel}+V_{1}\Psi+Q_{3}B_{\parallel}&=&0\nonumber\\
	V_{1}\Phi_{\parallel}+(V_{1}+\frac{V_{2}}{\Omega^{2}})\Psi+V_{3}B_{\parallel}&=&0\\
	-\frac{\beta_{i\parallel}}{2}(Q_{3}\Phi_{\parallel}+V_{3}\Psi)+A_{3}B_{\parallel}&=&0.\nonumber
\end{eqnarray}
Here, the fields have been normalized as
\begin{eqnarray}
\label{eq:normalization}
	\Phi=\frac{2e\delta\phi}{m_{i}v_{ti\parallel}^{2}},\quad \Psi=\frac{2e\delta\psi}{m_{i}v_{ti\parallel}^{2}},\quad \frac{\delta B_{\parallel}}{B}=B_{\parallel}
\end{eqnarray}
and $\Phi_{\parallel}=\Phi-\Psi$ is related to the parallel electric field. 
The associated coefficients, meanwhile, are given by
\begin{eqnarray}
\label{eq:coef}
	Q_{1}&=&-\sum_{s}\frac{T_{i\parallel}}{T_{s\parallel}}[ (1+\xi_{s}Z_{s} \Gamma_{0s})+a_{s}(1-\Gamma_{0s})],\nonumber\\
	Q_{3}&=&\sum_{s}\frac{|q_{s}|}{q_{s}}\frac{\Gamma_{0s}-\Gamma_{1s}}{1+a_{s}}(a_{s}-\xi_{s}Z_{s}),\nonumber\\
	V_{1}&=&-\sum_{s}(1+a_{s})\frac{T_{i\parallel}}{T_{s\parallel}}(1-\Gamma_{0s}),\nonumber\\
	V_{2}&=&\sigma_{k}(1+a_{i})b_{i},\nonumber\\
	V_{3}&=&\sum_{s}\frac{|q_{s}|}{q_{s}}(\Gamma_{0s}-\Gamma_{1s}),\nonumber\\
	A_{3}&=&-1+\sum_{s}\frac{\beta_{s\perp}}{1+a_{s}}(\Gamma_{0s}-\Gamma_{1s}) [\xi_{s}Z_{s}-a_{s}],
\end{eqnarray}
where $T_{e\parallel}/T_{i\parallel}=\tau$ is the parallel temperature ratio between electron and ion, $\beta_{s\perp(\parallel)}=8\pi N_{0}T_{s\perp(\parallel)}/B^{2}$, $\xi_{s}=\omega/|k_{\parallel}|v_{ts\parallel}$, $\Omega=\omega/|k_{\parallel}|v_{A}$, and $\omega$ is the desired eigenvalue. $q_{s}$ is the particle charge, $v_{A}=(B^{2}/4\pi N_{0}m_{i})^{1/2}$ is the Alfv\'{e}n velocity and $a_{s}=(T_{s\parallel}^{2}/T_{s\perp}^{2})-1$ explicitly accounts for  the temperature anisotropy.
$Z_{s}=Z(\xi_{s})$ is the plasma dispersion function.
	$\Gamma_{j s}=I_{j}(b_{s}) \textrm{exp}(-b_{s})$ can be regarded as finite Larmor radius (FLR) effect with $I_{j}$ being the first kind modified Bessel function with $b_{s}=k_{\perp}^{2}\rho_{ts\perp}^{2}/2$,
	and $\sigma_{k}=1-\sum_{s}\beta_{s\perp}a_{s}(1-\Gamma_{0s})/2 b_{s}$ is the gyrokinetic firehose stability term. 
	Furthermore,  from the symmetry property of plasma dispersion function, one readily concludes that the eigenvalues of Eq. (\ref{eq:eigenmode}) also occur in real conjugate  pairs.

	Adopting the same $\kappa^{2}_{c}, \epsilon_{i}$ and $\epsilon_{0}$ as in Tab.(\ref{tab:apr:14:14:12}). the present algorithm is applied to systematically study the nonlinear eigenvalue problem  Eq. (\ref{eq:eigenmode}). For the first time, the whole spectrum of normal modes are illustrated, including both the ion-sound wave (ISW) branch, the kinetic Alfv\'{e}n wave (KAW) branch and the mirror mode (MM) branch, as in Fig. (\ref{eps:kaw}).

	Unlike the shear Alfv\'{e}n wave in the ideal magnetohydrodynamic limit, the kinetic Alfv\'{e}n wave possesses finite parallel electric field due to the coupling between the shear Alfv\'{e}n wave and ion-sound wave branch \cite{hasegawa75,hasegawa76}.
	As a consequence, KAWs are expected to play crucial roles in heating, accelerating and transport processes of charged particles.
	The present algorithm, however, enables the identification of a new class of kinetic Alfv\'{e}n waves with finite parallel electric field, arising from the strongly coupling between shear Alfv\'{e}n wave and the mirror mode branch (see Fig. (\ref{eps:kaw-mm})). This type of KAW has distinctive features, and may play a crucial role in anisotropic high-$\beta$ solar wind plasmas \cite{chen21}.
	Note that here we just present the key results, interested readers are referred to the original work for details \cite{chen21}.

\begin{figure}[!htp]
\centering
\includegraphics[scale=0.45]{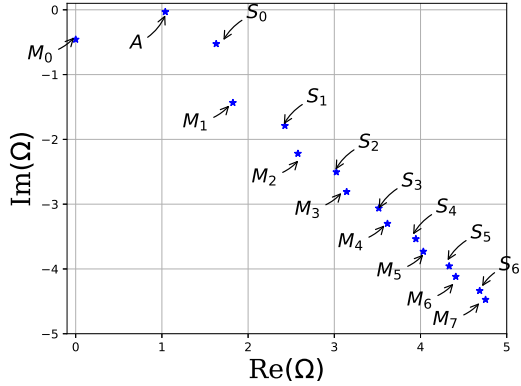}
	\caption{Eigenvalues of Eq. (\ref{eq:eigenmode}) in the region $\textrm{Im}(\Omega)\ge -5$ and $0\le \textrm{Re}(\Omega)\le 5$, for $\beta_{i\perp}=1$, $b_{i}=0.1$, $\tau=10$, $a_{i}=a_{e}=0$ and $m_{i}/m_{e}=1836$. KAW, mirror modes and ion-sound waves are, respectively, denoted by $A$, $S_{j}$ and $M_{j}$.  $\kappa^{2}_{c}, \epsilon_{i}$ and $\epsilon_{0}$ are the same as Table (\ref{tab:apr:14:14:12})}
\label{eps:kaw}
\end{figure}
\begin{figure}[!htp]
\centering
\includegraphics[scale=0.45]{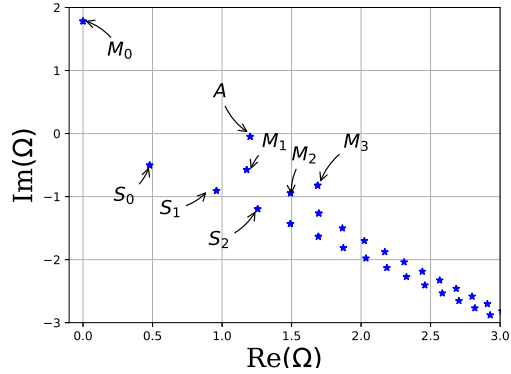}
	\caption{Eigenvalues of Eq. (\ref{eq:eigenmode}) in the region $\textrm{Im}(\Omega)\ge -3$ and $0\le \textrm{Re}(\Omega)\le 3$, for $\beta_{i\perp}=1$, $b_{i}=0.1$, $\tau=1$, $a_{i}=0$, $a_{e}=-0.82$ and $m_{i}/m_{e}=1836$. KAW, mirror modes and ion-sound waves are, respectively, denoted by $A$, $S_{j}$ and $M_{j}$.  $\kappa^{2}_{c}, \epsilon_{i}$ and $\epsilon_{0}$ are the same as Table (\ref{tab:apr:14:14:12})}
\label{eps:kaw-mm}
\end{figure}


\section{Conclusions}
\label{sec:feb:09:15:15}
In this paper, we revisited the numerical method for locating the zeros and poles of a meromorphic function based on the generalized argument principle.
After a detailed sensitivity analysis, a subdivision-transformation-calculation scheme is proposed to ensure the algorithm stability. 
Contrary to previous methods, this algorithm gives a novel method to automatically estimate the underlying numerical errors.
Numerical examples are presented to validate and verify the algorithm and related error estimates.
Meanwhile, the algorithm is further applied to investigate linear waves arising from plasma physics, which are essentially proper subsets of the nonlinear eigenvalue problem.
Especially, the algorithm has provided the first whole spectrum of waves in uniform gyrokinetic plasmas, and led to the discovery of a new class of kinetic Alfv\'{e}n waves.
Detailed applications of the present algorithm to various waves and instabilities in plasma physics will be reported in future publications.




\begin{acknowledgements}
The author would like to thank Xiaoke Fang and Prof. Liu Chen for useful conversations.
This work was supported by National Natural Science Foundation of China under Grant No. 11905097.
The support from the European Research Council (ERC) under the European Unions Horizon 2020 research and innovation programme (grant agreement No. 805162) is also gratefully acknowledged.

\end{acknowledgements}

\section*{Data Availability Statement}
Data sharing not applicable to this article as no datasets were generated or analysed during the current study.




%
 \section*{Conflict of interest}
 The authors declare that they have no conflict of interest.



\end{document}